\documentclass[12pt]{article}
\usepackage{geometry}                
\usepackage[parfill]{parskip}    
\usepackage{graphicx}
\usepackage{amssymb}
\usepackage{epstopdf}
\usepackage{yfonts}
\usepackage{bbold}
\usepackage{comment}
\usepackage{amsmath,amscd}
\usepackage{amsthm}
\usepackage{mathrsfs}

\title{$\bb{C}$-Graded Vertex Algebras and Conformal Flow}
\author{Rob Laber and Geoffrey Mason\thanks{Supported by NSF}\\
Department of Mathematics, UC Santa Cruz}
\date{}
\newcommand{\bb}[1]{\mathbb{#1}}
\newcommand{\mf}[1]{\mathfrak{#1}}
\newcommand{\wt}[1]{\overline{|#1|}}

\newtheorem{thm}{Theorem}[section]
\newtheorem{prop}[thm]{Proposition}
\newtheorem{lem}[thm]{Lemma}
\newtheorem{cor}[thm]{Corollary}

\theoremstyle{definition}
\newtheorem*{definition*}{Definition}

\theoremstyle{theorem}
\newtheorem*{thm*}{Theorem}

\theoremstyle{definition}
\newtheorem*{remark*}{Remark}

\bigskip
\begin{document}

\maketitle

\begin{abstract}
\noindent
We consider \emph{$\mathbb{C}$-graded vertex algebras}, which are vertex algebras $V$ with a $\mathbb{C}$-grading such that $V$ is an admissible
$V$-module generated by  `lowest weight vectors'.\ We show that such vertex algebras have a `good' representation theory in the sense that there is a Zhu algebra $A(V)$ and a bijection between simple admissible $V$-modules and simple $A(V)$-modules.\ We also consider \emph{pseudo vertex operator algebras}, which are $\mathbb{C}$-graded vertex algebras with a conformal vector such that the homogeneous subspaces of $V$ are generalized 
eigenspaces for $L(0)$; essentially, these are VOAs that lack any semisimplicity or integrality assumptions  on $L(0)$.\ As a motivating example, we show that deformation of the conformal structure (conformal flow) of a strongly regular VOA (eg a lattice theory, or WZW model) is a path
in a space whose points are PVOAs.
\end{abstract}

\section{Introduction}

The motivation for the results in this paper arise from several related sources.\ The first concerns
\emph{conformal flow}, or \emph{conformal deformation}.\ If $V=(V, Y, \mathbf{1}, \omega)$ is a VOA then, as is well-known,
one can often \emph{deform} $\omega$ to obtain another conformal vector $\omega'$ so that
$V'=(V, Y, \mathbf{1}, \omega')$ is again a VOA.\ Indeed, $\omega$ and $\omega'$ may be considered
as the endpoints
of a continuous path in $\mathbb{C}$ in which intermediate points correspond to
vectors $\omega''\in V$ which are conformal in the sense that the modes of $Y(\omega'', z)$
generate an action of the Virasoro algebra on $V$ (the central charge varies with $\omega''$),
but which in general do not confer the structure of a VOA on $(V, Y, \mathbf{1}, \omega'')$.\ It is thus natural to ask exactly what sort of a structure does $(V, Y, \mathbf{1}, \omega'')$ support in general?\
We will show that for nice enough VOAs, i.e., strongly regular VOAs (technical definitions will be given below) each intermediate $(V, Y, \mathbf{1}, \omega'')$ is what we call a \emph{pseudo vertex operator algebra} (PVOA).\ In this way, conformal flow becomes
a path in a space of PVOAs, all of which come from the same mother vertex algebra.

\medskip
A second motivation comes from the study of $C_2$-cofinite, or logarithmic vertex operator algebras.\ These are VOAs $V$ in the usual sense, however $L(0)$ 
does not necessarily act semisimply on $V$-modules.\ Interesting examples such as the triplet algebras   
(e.g., \cite{AM}) arise as  subspaces of a conformal deformation of a lattice theory, i.e.,  subalgebras of a PVOA.\ Given the importance of these objects, it is natural to relax the axioms so that $L(0)$ is not required to act semisimply on any module, even $V$ itself.\ This leads to our notion of $\mathbb{C}$-graded vertex algebra, which is formulated so that it puts no requirements on the conformal vector 
(including existence!) yet has a good representation theory exemplified by its Zhu algebra.\ Having said this, we emphasize that the relationship between our $\mathbb{C}$-graded vertex algebras and other kinds of theories remains unclear to us.

\medskip
Theories resulting from conformal deformation of lattice VOAs were studied systematically in
\cite{DM}, under the name of \emph{shifted} VOAs.\ They were
more-or-less shown to be PVOAs, but by an argument that does not generalize.\  Conformal deformation was also used in \cite{L} as a way to classify conformal vectors.
The desire to put these results into a more general and theoretical context 
provided additional motivation.\ 

\medskip
The paper is organized as follows.\ In Section 2 we introduce
$\mathbb{C}$-graded vertex algebras and their modules.\ In Section 3 we discuss the Zhu algebra 
$A(V)$ of a $\mathbb{C}$-graded vertex algebra $V$ and establish a bijection between simple 
admissible $V$-modules and simple modules over $A(V)$.\ This follows the general lines of argument
in \cite{DLM2}, \cite{Z}, but with notable differences resulting from our weaker axiomatic set-up.
In Section 4 we introduce PVOAs and prove  that  conformal deformations
of strongly regular VOAs are PVOAs.\ We also show that the Zhu algebra of any shifted lattice theory is semisimple.

\section{$\bb{C}$-Graded Vertex Algebras}

\subsection{Notation and Definitions}

\begin{definition*} A \emph{vertex algebra} is a quadruple $V=(V,Y,\bb{1}, T)$ consisting of a vector space $V$, a linear map
\begin{eqnarray*}Y:V&\to& End(V)[[z,z^{-1}]]\\
a&\mapsto&Y(a,z) = \sum_{n\in\bb{Z}}a(n)\;z^{-n-1},
\end{eqnarray*} a distinguished vector $\bb{1}\in V$ called the \emph{vacuum vector}, and an operator $T\in End(V)$ with the property that 
\[ [T,Y(a,z)] = \frac{\partial}{\partial z} Y(a,z) = Y(Ta, z).\]  There are various additional requirements on this data, and we refer  to \cite{LL} or \cite{K} for more details.\ We often simply refer to a vertex algebra as $V$. \end{definition*}

\bigskip
\begin{definition*} A vertex algebra $V$ is said to be $\bb{C}$-\emph{graded} if $V$ satisfies the following two properties:

\quad i) $V$ is a direct sum 
\[V=\bigoplus_{\mu\in\bb{C}}V_{\mu},\]
such that for any homogeneous element $a\in V_{\lambda}$, one has
\begin{equation}\label{cgraded}
a(n)V_{\mu}\subseteq V_{\mu+\lambda -n-1}.\end{equation}

\quad (ii) $V$ is generated by a set of \emph{lowest weight vectors}, where a lowest weight vector is a homogeneous vector $v\in V_{\mu}$ that satisfies the following:  For any $a\in V_{\lambda}$, if $a(n)v\neq 0$, then either $n=\lambda -1$ or $n<Re(\lambda) -1$.

This completes the definition.
\end{definition*}
\bigskip

An element $a\in V_{\mu}$ is said to have \emph{weight} $\mu$, and we denote this by $|a| = \mu$.  We define the operator $L\in End(V)$ as the linear extension of the map\begin{eqnarray*}
V_{\mu}&\to&V_{\mu}\\
a&\mapsto& \mu a = |a|a.
\end{eqnarray*}

As a straightforward application of the creativity axioms (see \cite{K} or \cite{FLM}) and formula \eqref{cgraded}, we have:
\bigskip
\begin{prop} Let $V=\bigoplus_{\mu\in\bb{C}}V_{\mu}$ be a $\bb{C}$-graded vertex algebra.  Then

\quad (i) $\bb{1}\in V_0$.

\quad (ii) If $a\in V_{\mu}$ then $Ta\in V_{\mu +1}$.\end{prop}

\subsection{Modules}

Here we define the notion of module over a $\bb{C}$-graded vertex algebra $V$.  Since $V$ is a vertex algebra, we already have the notion of \emph{weak} $V$-\emph{module} (see \cite{LL}).
\bigskip
\begin{definition*} An \emph{admissible} $V$-\emph{module} is a weak $V$-module $M$ with a grading of the form
\[M= M(0) \oplus\bigoplus_{\substack{\mu\in\bb{C}\\Re(\mu)>0}} M(\mu)\]
such that $M(0)\neq 0$ and for any homogeneous $a\in V_{\lambda}$, one has
\begin{equation}\label{admissible}
a(n)M(\mu)\subseteq M(\mu+\lambda -n-1).\end{equation}
This completes the definition.
\end{definition*}

We note that any $\bb{C}$-graded vertex algebra $V$ is an admissible module over itself.  Since $V$ is generated by its lowest weight vectors, every element $v$ in $V$ is a sum of elements of the form
\[a^r(n_r)a^{r-1}(n_{r-1}).. .a^1(n_1)w\]
where $a^i$ is homogeneous element and $w$ is a lowest weight vector.  We let $V(0)$ be the space of lowest weight vectors, and we say that \[deg\left(a^r(n_r)a^{r-1}(n_{r-1}).. .a^1(n_1)w\right) = \sum_{i=1}^r\left(|a^i| - n_i -1\right)\]
One can use the weak commutativity (\cite{LL}) and the fact that $w$ is a lowest weight vector to show that either \[Re\left(\sum_{i=1}^r\left(|a^i| - n_i -1\right)\right) > 0\]
or
\[\left(\sum_{i=1}^r\left(|a^i| - n_i -1\right)\right) =0. \]
Thus, we define $V(\lambda)$ to be the space of all $v\in V$ with $deg(v) = \lambda$, and this gives $V$ the structure of an admissible $V$-module.

\section{The Zhu Algebra $A(V)$}

Throughout this section we let $V=(V,Y,\bb{1},T)$ be a $\bb{C}$-graded vertex algebra.  Here we fix some notation.  For $a\in V_{\mu}$, we let $\wt{a}$ denote the ceiling of the real part of $\mu$, ie,
\[ \wt{a} := min\{\;n\in\bb{Z}\;|\;n\geq Re(\mu)\;\}.\]
 Of course, $La = \wt{a}a$ if and only if $\mu\in\bb{Z}$, or equivalently, if and only if $|a| = \wt{a}$.  Let $V^r$ be the set of all elements $a\in V$ with $r= |a| -\wt{a}$. Note that $V = \bigoplus_{r\in\bb{C}}V^r$.

\subsection{Construction of $A(V)$}

\begin{definition*} Let $a\in V^r$.  Define the circle and star products on $V$ as the linear extensions of the following:
\[a\circ b := Res_z\dfrac{(1+z)^{\wt{a} +\delta_{r,0} -1}}{z^{1+\delta_{r,0}}}Y(a,z)b \]
and
\[a\star b := \delta_{r,0}Res_z\dfrac{(1+z)^{|a|}}{z}Y(a,z)b. \]

We define $O(V)$ to be the linear span of all element of the form $a\circ b$ for $a,b\in V$.
\end{definition*}

\bigskip
\begin{lem} If $r\neq 0$, then $V^r\subseteq O(V)$.
\end{lem}

\begin{proof}   Let $a\in V^r$.  Then
\[a = Res_z\dfrac{(1+z)^{\wt{a} -1}}{z}Y(a,z)\bb{1} = a\circ \bb{1} \in O(V).\]\end{proof}
\bigskip
\begin{lem} For any homogeneous $a\in V$, one has
\begin{equation}
(T+L)a \equiv 0\quad\quad mod\;O(V).\label{modov}
\end{equation}

In particular, $Ta\equiv -La\quad (mod\;O(V))$.\end{lem}

\begin{proof} If $a\in V^r$ for some nonzero $r$, then the result follows because both sides of \eqref{modov} are in $V^r\subseteq O(V)$.  If $a\in V^0$, then this follows immediately from
\[(T+L)a  = Ta + |a|a = Res_z\dfrac{(1+z)^{|a|}}{z^2}Y(a,z)\bb{1}  \in O(V),\]
where we use the fact that $Ta = a(-2)\bb{1}$.\end{proof}

\bigskip
\begin{prop} For any $a\in V^r$ and any $m\geq n\geq 0$, we have 
\begin{equation}Res_z \dfrac{(1+z)^{\wt{a} +\delta_{r,0} -1+n}}{z^{1+\delta_{r,0} +m}} Y(a,z)b \;\in O(V).\label{resainVr}\end{equation}
\end{prop}
\begin{proof} The proof is the same as \cite{Z} Lemma 2.1.2 if one replaces $L(-1)$ by $T$.\end{proof}

\bigskip
\begin{lem} Let $a,b\in V$ be homogeneous elements.  Then 
\begin{equation}
Y(a,z)b\equiv (1+z)^{-|a|-|b|}Y(b,\frac{-z}{1+z})a\quad mod\;O(V).\label{commlemma}
\end{equation}
\end{lem}

\begin{proof} This result was proved in \cite{Z} during the proof of Lemma 2.1.3, but we provide additional details here.  Using the skew symmetry, we have
\begin{eqnarray*}
Y(a,z)b &=& e^{zT}Y(b,-z)a\\
&=& \sum_{i\in\bb{Z}}e^{zT}b(i)a\;z^{-i-1}\\
&=&\sum_{i\in\bb{Z}}\sum_{j\geq 0} \frac{z^jT^j}{j!}b(i)a\;z^{-i-1}.
\end{eqnarray*}

Recall that $La +Ta\in O(V)$ for any $a\in V$, so we have the congruence $(mod\;O(V))$
\begin{eqnarray*}
 \frac{z^jT^j}{j!}b(i)a &=& \frac{z^j}{j!} (-L)T^{j-1}b(i)a\\
 &\equiv &(-1)\frac{z^j}{j!}(|a|+|b|-i-1+j-1)T^{j-1}b(i)a\\
 &\equiv &(-1)\frac{z^j}{j!}(|a|+|b|-i-1+j-1)(-L)T^{j-2}b(i)a\\
&\equiv & (-1)^j \frac{z^j}{j!}(|a|+|b|-i-1+j-1)\dots (|a|+|b|-i-1)b(i)a.
\end{eqnarray*}

Returning to the previous calculation, we have $(mod\;O(V))$,
\begin{eqnarray*}
Y(a,z)b &\equiv &\sum_{i\in\bb{Z}}\sum_{j\geq 0}  (-1)^j \frac{z^j}{j!}(|a|+|b|-i-1+j-1)\dots (|a|+|b|-i-1)b(i)a\;z^{-i-1}\\
& = & \sum_{i\in\bb{Z}}(-1)^{i+1} b(i)a\;(1+z)^{-|a|-|b|+i+1}z^{-i-1}\\
&=& (1+z)^{-|a|-|b|}Y(b,\frac{-z}{1+z})a.
\end{eqnarray*}\end{proof}

\begin{lem} If $a$ and $b$ are homogeneous elements in $V^0$, then we have the identities
\begin{equation}
a\star b \equiv Res_z\dfrac{(1+z)^{\wt{b}-1}}{z}Y(b,z)a\quad mod\;O(V)\label{comm1}
\end{equation}
and
\begin{equation}
a\star b -b\star a \equiv Res_z\;(1+z)^{\wt{a}-1}Y(a,z)b\quad mod\;O(V).\label{comm2}
\end{equation}
\end{lem}

\begin{proof} See \cite{Z} Lemma 2.1.3, and use \eqref{commlemma}\end{proof}

\bigskip
\begin{prop} $O(V)$ is a left ideal of $V$ with respect to the star product $``\star "$.\end{prop}

\begin{proof} We must show that $a\star (b\circ c) \in O(V)$ for any $a,b,c\in V$.  Note that if $a\in V^r$ and $r\neq 0$, then $a\star (b\circ c) = 0\in O(V)$.  Therefore, we may assume that $a\in V^0$.  If $b\in V^0$, then the result is true by Theorem 2.1.1 of \cite{Z}.  The only remaining case is when $a\in V^0$ and $b\in V^r$ for some $r\neq 0$.  We calculate
\begin{eqnarray*}
a\star(b\circ c) - b\circ (a\star c) &=& Res_zRes_w\;Y(a,z)Y(b,w)c\dfrac{(1+z)^{|a|}}{z}\dfrac{(1+w)^{\wt{b} -1}}{w}\\
&&\hspace{.2 in} - Res_wRes_z\;Y(b,w)Y(a,z)c\dfrac{(1+z)^{|a|}}{z}\dfrac{(1+w)^{\wt{b} -1}}{w}\\
&=& Rez_wRes_{z-w}\;Y(Y(a,z-w)b, w)c\dfrac{(1+z)^{|a|}}{z}\dfrac{(1+w)^{\wt{b} -1}}{w}\\
&=&Res_wRes_{z-w}\sum_{i,j=0}^{\infty}{|a|\choose i}(-1)^j(z-w)^{i+j}\dfrac{(1+w)^{\wt{b}+|a|-1-i}}{w^{2+j}}\\
&&\hspace{.2 in}\cdot \;Y(Y(a,z-w)b,w)c\\
&=&\sum_{i,j=0}^{\infty}{|a|\choose i}(-1)^jRes_w\dfrac{(1+w)^{\wt{b}+|a|-1-i}}{w^{2+j}}Y(a(i+j)b,w)c\\
&=&\sum_{i,j=0}^{\infty}{|a|\choose i}(-1)^jRes_w\dfrac{(1+w)^{\wt{a(i+j)b} +j}}{w^{2+j}}Y(a(i+j)b,w)c
\end{eqnarray*}

which is a sum of terms in $O(V)$ by \eqref{resainVr}.  Since $b\circ(a\star c) $ is clearly in $O(V)$, and $a\star (b\circ c) - b\circ(a\star c) \in O(V)$ by the previous calculation, we see that $a\star (b\circ c) \in O(V)$, which proves that $O(V)$ is a left ideal of $V$ with respect to the star product.\end{proof}
\bigskip
\begin{prop} $O(V)$ is a right ideal of $V$.\end{prop}

\begin{proof} We must show that any element of the form $(a\circ b)\star c$ is in $O(V)$.  It suffices to prove this for homogeneous $a,b,c$.  By the definition of the star product we know that if $a\circ b\in V^r$ for some nonzero $r$, then $(a\circ b)\star c =0\in O(V)$.  Moreover, if $a$ and $b$ are both in $V^0$, then the result follows from \cite{Z}.  Thus, we assume that $a\in V^r$ for some nonzero $r$ and $a\circ b \in V^0$.  It is easy to see that $V^0\star V^r\subseteq V^r$ for any $r$.  Therefore, if $c\in V^r$ for some nonzero $r$, then $(a\circ b)\star c \in V^r\subseteq O(V)$.  Thus, we are reduced to the case where $a\circ b\in V^0$ and $c\in V^0$.  By \eqref{comm1} it suffices to show that 
\[x=Res_z\dfrac{(1+z)^{|c|-1}}{z}Y(c,z)(a\circ b) \in O(V).\]
We calculate
\begin{eqnarray*}
x &=& Res_zRes_w\dfrac{(1+z)^{|c|-1}(1+w)^{\wt{a}-1}}{zw}Y(c,z)Y(a,w)b\\
&=& Res_wRes_{z-w}Y(Y(c,z-w)a,w)b\;\dfrac{(1+z)^{|c|-1}(1+w)^{\wt{a}-1}}{zw}\\
&&\hspace{.2 in} +Res_wRes_zY(a,w)Y(c,z)b\;\dfrac{(1+z)^{|c|-1}(1+w)^{\wt{a}-1}}{zw}\\
\end{eqnarray*}
Since $a\not\in V^0$, we see that the second term in the above sum is in $O(V)$.  Therefore we have ($mod\;O(V)$):
\begin{eqnarray*}
x &\equiv& Res_wRes_{z-w}Y(Y(c,z-w)a,w)b\;\dfrac{(1+z)^{|c|-1}(1+w)^{\wt{a}-1}}{zw}\\
&=&\sum_{i=0}^{\infty}\sum_{j=0}^{\infty}{|c|-1\choose i}Res_wRes_{z-w}Y(Y(c,z-w)a,w)b\;\dfrac{(z-w)^{i+j}(1+w)^{\wt{a}+|c|-2-i}}{w^{2+j}}\\
&=&\sum_{i=0}^{\infty}\sum_{j=0}^{\infty}{|c|-1\choose i}Res_wY(c(i+j)a,w)b\;\dfrac{(1+w)^{\wt{a}+|c|-2-i}}{w^{2+j}}\\
&=&\sum_{i=0}^{\infty}\sum_{j=0}^{\infty}{|c|-1\choose i}Res_wY(c(i+j)a,w)b\;\dfrac{(1+w)^{\wt{c(i+j)a} -1+j}}{w^{2+j}}\\
\end{eqnarray*}
Again, since $c\in V^0$ and $a\not\in V^0$, we see that $c(i+j)a\not\in V^0$, and we use \eqref{resainVr} to see that above is a sum of terms in $O(V)$.  We conclude that $(a\circ b)\star c \in O(V)$ for any $a,b,c\in V$, hence $O(V)$ is a right ideal of $V$.\end{proof}

\begin{thm} Define $A(V):= V/O(V)$.  Then $A(V)$ is an associative algebra with respect to the $``\star "$ product.\end{thm}

\begin{proof}   We may assume that $a,b,c\in V^0$, since otherwise  $a\star (b\star c) = (a\star b)\star c=0$ in $A(V)$.  However, this result was proved in \cite{Z}.\end{proof}

\subsection{The functor $\Omega$}

The goal of this section is to construct a functor $\Omega$ from the category of admissible $V$-modules to the category of $A(V)$-modules.  Throughout, we let $M$ be an admissible $V$-module.
\bigskip
\begin{definition*}We define the \emph{vacuum space} $\Omega (M)$ of $M$ to be the set of all lowest weight vectors of $M$.\end{definition*}

Note that if $M$ is a simple admissible $V$-module, then one has $\Omega(M) = M(0)$.

\bigskip
\begin{remark*} Recall from \eqref{admissible} that for any admissible $V$-module $M$, and any homogeneous element $a\in V$, one has
\[a(n)M(\mu)\subseteq M(|a| +\mu -n-1).\]
From this, it is evident that not every element $a\in V$ has a mode that acts on the graded subspaces of $M$.   Indeed, if \[|a|+\mu -n-1 = \mu ,\] we see that $|a|=n+1$.  In particular, $|a|\in\bb{Z}$, which is only true for those elements $a\in V^0$.\end{remark*}
\bigskip
\begin{definition*} Let $a\in V$ be a homogeneous element.  We define the \emph{zero mode} $o(a)$ of $a$ as \[o(a) = a(\wt{a} -1).\]
We extend this definition to all of $V$ by linearity.\end{definition*}

\quad Note that if $a\in V^r$ for some nonzero $r$, then $\wt{a}-1> Re(|a|)-1$, and so $o(a) = a(\wt{a}-1)$ annihilates any lowest weight vector.  On the other hand, if $a\in V^0$, then $\wt{a} = |a|$, so this definition of $o(a)$ reduces to the definition given in \cite{Z}.
\bigskip
\begin{thm} Let $M$ be a simple admissible $V$-module with lowest weight space $M(0)$.  Then there is an action of the associative algebra $A(V)$ on $M(0)$, with $a\in A(V)$ acting via its zero mode $o(a)$.\end{thm}

\begin{proof} We must show that $O(V)$ annihilates $M(0)$, and that $o(a\star b) = o(a)o(b)$ for $a,b\in A(V)$.  First, we show that $O(V)$ annihilates $M(0)$.  By the previous discussion, we see that $V^r$ annihilates $M(0)$ whenever $r\neq 0$.  Therefore, we must show that $V^0\cap O(V)$ annihilates $M(0)$.  For this it suffices to show that 
\[o(a\circ b) M(0) = 0\]
for $a\circ b\in V^0$.  If $a,b\in V^0$, this result was proved in \cite{Z}.  Thus, we assume that $a\in V^r$ for some nonzero $r$, and $a\circ b\in V^0$.  Note that in this case, $b\not\in V^0$, and we have 
\[\wt{a(n)b} = \wt{a}+\wt{b} -n-2\]
for any $n\in\bb{Z}$.  Now let $w\in M(0)$.  Following \cite{DLM2}, we use a property of the $\delta$ function to rewrite the Jacobi Identity as 
\begin{eqnarray*}
z_1^{-1}\delta\left(\dfrac{z_0+z_2}{z_1}\right)Y(a, z_1)Y(b,z_2)w -z_0^{-1}\delta\left(\dfrac{z_2-z_1}{-z_0}\right)Y(b, z_2)Y(a,z_1)w \\ 
= z_1^{-1}\delta\left(\dfrac{z_2+z_0}{z_1}\right)Y(Y(a, z_0)b, z_2)w
\end{eqnarray*}
Since $a,b\not\in V^0$, we know that $a(\wt{a}-1)w = b(\wt{b}-1)w=0$.  Thus, $Rez_{z_1}Res_{z_2}z_1^{\wt{a}-1}z_2^{\wt{b}-1}$ of the left hand side is equal to $0$.  Then we have
\begin{eqnarray*}
0&=&Rez_{z_1}Res_{z_2}z_1^{\wt{a}-1}z_2^{\wt{b}-1}z_1^{-1}\delta\left(\dfrac{z_2+z_0}{z_1}\right)Y(Y(a, z_0)b, z_2)w\\
&=&Res_{z_2}z_2^{\wt{b}-1}(z_2+z_0)^{\wt{a}-1}Y(Y(a, z_0)b, z_2)w\\
&=& Res_{z_2}\sum_{i=0}^{\infty}{\wt{a}-1\choose i} z_0^iz_2^{\wt{a}+\wt{b} -2-i}Y(Y(a, z_0)b, z_2)w\\
&=&Res_{z_0}z_0^{-1}Res_{z_2}\sum_{i=0}^{\infty}{\wt{a}-1\choose i} z_0^iz_2^{\wt{a}+\wt{b} -2-i}Y(Y(a, z_0)b, z_2)w\\
&=&\sum_{i=0}^{\infty}{\wt{a}-1\choose i} (a(i-1)b)(\wt{a}+\wt{b}-2-i)w\\
&=& o(a\circ b)w.
\end{eqnarray*}

Altogether, this shows that $O(V)$ annihilates $M(0)$.  To see that $o(a\star b) = o(a)o(b)$ on $M(0)$, observe that if $a\not\in V^0$, then $o(a)$ annihilates $M(0)$, and $a\star b = 0$, so the result holds.  If $a\in V^0$ and $b\not\in V^0$, then $a\star b\in O(V)$, and therefore $o(a\star b)  = 0=o(b)=o(a)o(b)$ on $M(0)$.  The only remaining case is when $a,b\in V^0$, and this case was proved in \cite{Z}.\end{proof}

The following corollary is the main result of this section.
\bigskip
\begin{cor} Let $M$ be any admissible $V$-module.  Then there is an action of $A(V)$ on $\Omega(M)$.  In particular, $\Omega$ defines a functor from the category of admissible $V$-modules to the category of $A(V)$-modules.\end{cor}

\begin{proof} The action of $A(V)$ on $\Omega(M)$ follows from the fact that $\Omega(M)$ is a sum of lowest weight spaces of $M$, hence a sum of modules for $A(V)$.  The remaining functorial properties of $\Omega$ are clear.\end{proof}

\begin{prop}\label{prop3.12} Let $M$ be a simple admissible $V$-module.  Then $\Omega(M) = M(0)$ is a simple $A(V)$-module.\end{prop}

\begin{proof} See \cite{DLM2} Proposition 5.4\end{proof}

\subsection{The Lie Algebra $V_{Lie}$}

In \cite{DLM2} we encounter a Lie algebra which plays a role in their construction of admissible $V$-modules from $A_g(V)$-modules.  Here we use this technique to construct a Lie algebra $V_{Lie}$ which plays an analogous role.  The construction of $V_{Lie}$ is similar to the construction found in \cite{DLM2}, but with some notable differences.

\quad Let $t$ be an indeterminate, and set $\mathscr{L}(V) = \bb{C}[t,t^{-1}]\otimes V$.  There is a natural vertex algebra structure on $\bb{C}[t,t^{-1}]$, and so $\mathscr{L}(V)$ is a vertex algebra, as it is a tensor product of vertex algebras.  We have the translation covariance operator
\[D = \frac{d}{dt}\otimes 1 + 1\otimes T,\]
and it is well known that the space
\[\mathscr{L}(V)/D\mathscr{L}(V)\] carries the structure of a Lie algebra, with Lie bracket given by 
\[[a+D\mathscr{L}(V), b+D\mathscr{L}(V)] = a(0)b+D\mathscr{L}(V).\]
Therefore, we define 
\[V_{Lie} = \mathscr{L}(V)/D\mathscr{L}(V).\]

We let $a_n$ denote the image of $t^n\otimes a$ in $V_{Lie}$.  Then note that 
\begin{equation} \left[a_n,b_m\right] = \sum_{i=0}^{\infty}{n\choose i} (a(i)b)_{m+n-i}.\label{LieCommRlns}
\end{equation}

\quad We give a complex grading to $V_{Lie}$ by declaring, for homogeneous $a\in V$, 
\[deg(a_n) = |a|-n-1.\]
We then have the Lie subalgebras
\[ (V_{Lie})^+ = \left< \;a_n\;|\; Re(deg(a_n))>0\;\right>\]
and
\[(V_{Lie})^0 = \left<\;a_n\;|\;deg(a_n) = 0\;\right>.\]
The space
\[\left(V_{Lie}\right)^- = \left< a_n\;|\; Re(deg(a_n)) <0\quad \text{or}\quad Re(deg(a_n))=0\;\text{and} \;Im(deg(a_n))\neq 0 \;\right>\]
is not a Lie subalgebra of $V_{Lie}$.  Indeed, if $|a|=n+i\lambda$ for some $n\in\bb{Z}$ and nonzero $\lambda \in \bb{R}$, and $|b| = n-i\lambda$, then $a_{\wt{a}-1} =a_{n-1}$ and $b_{\wt{b}-1}=b_{n-1} $ are in $(V_{Lie})^-$, but 
\begin{equation*}
[a_{n-1},b_{n-1}] = \sum_{i=0}^{\infty} { n-1 \choose i} (a(i)b)_{2n -2-i} \in (V_{Lie})^0.
\end{equation*}

Therefore, we have that 
\[(V_{Lie})^{\leq 0}=(V_{Lie})^0 \oplus (V_{Lie})^-\]
is a Lie subalgebra of $V_{Lie}$.

\quad Of course, it is easy to see that $(V_{Lie})^0$ is spanned by elements of the form $a_{|a|-1}$ for homogenous $a\in V^0$.  This gives a surjection 
\[ V^0\twoheadrightarrow(V_{Lie})^0\]

\begin{prop} The kernel of the map 
\begin{eqnarray*}
V^0&\to&(V_{Lie})^0\\
a&\mapsto&a_{|a|-1}
\end{eqnarray*}
is $(T+L)V^0$. \end{prop}

\begin{proof} It is easy to see that $(T+L)V^0$ is contained in the kernel.  Now suppose $a\in V^0$ is in the kernel.  The zero element in the quotient space $(V_{Lie})^0$ is equal to the space
\[D\left< a_n\;|\;deg(a_n) = -1\;\right>,\]
and therefore the image of $a$ must be a finite sum of the form
\[ \sum_{j} D(b^j_{j}) = \sum_{j}j(b^j_{j-1}) + (Tb^j)_{j}.\]
for some homogeneous $b^j\in V_j$.  If we write $a$ as a sum of homogeneous terms
\[a=\sum_{i}a^i\]
with $a^i\in V_i$, then we simply solve for each $a^i$ to find that 
\[a^i = Tb^{i-1} +Lb^i\]
and therefore 
\[a = (T+L)\sum_ib^i.\]\end{proof}
\quad This proves the claim, and further shows that we have an isomorphism of Lie algebras
\[V^0/(T+L)V^0 \cong (V_{Lie})^0,\]
where the Lie structure on $V^0/(T+L)V^0$ is inherited via this linear isomorphism.

The next lemma was proved in \cite{DLM2}.
\bigskip
\begin{lem}\label{lielemma} The map 
\begin{eqnarray*}
(V_{Lie})^0&\to& A(V)_{Lie}\\
a_{|a|-1}&\mapsto& a+O(V)
\end{eqnarray*}
is a Lie algebra epimorphism.\end{lem}

\begin{proof} First observe that we have the following inclusions:
\[(T+L)V^0 \subseteq O(V^0)\subseteq O(V)\cap V^0.\]
Then we have the corresponding epimorphisms
\[(V_{Lie})^0 \cong V^0/\left((T+L)V^0\right) \twoheadrightarrow A(V^0) \cong V^0 /O(V^0) \twoheadrightarrow A(V)\cong V^0/(O(V)\cap V^0)\]
which induce the desired linear epimorphism.  To see that $a_{|a|-1}\mapsto a+O(V)$ is a Lie algebra morphism, use \eqref{comm2} and \eqref{LieCommRlns}.\end{proof}

\subsection{The $\Lambda$ functor}

In this section we provide a natural construction of an admissible $V$-module from an $A(V)$-module, thereby giving a functor $\Lambda$ from the category of $A(V)$-modules to the category of admissible $V$-modules.  Throughout this section, we let $U$ be a module for the associative algebra $A(V)$.

\quad Since $U$ is a module for $A(V)$, and hence for $A(V)_{Lie}$, we can lift $U$ to a module for $(V_{Lie})^0$ via
\[\begin{array}{ccccc}
(V_{Lie})^0& \twoheadrightarrow& A(V)_{Lie} &\to & End(U)\\
a_{|a|-1}&\mapsto& a+O(V) &\mapsto& o(a).
\end{array}\]
Thus, $a_{|a|-1}\in (V_{Lie})^0$ acts as $o(a)$ on $U$.  We want to extend this action to \[(V_{Lie})^{\leq 0}=(V_{Lie})^0\oplus (V_{Lie})^-,\] and we do so by letting $(V_{Lie})^-$ annihilate $U$.  It is not immediately clear that this definition does in fact yield an action of $(V_{Lie})^{\leq 0}$, however, the next result shows that this is indeed the case.

\bigskip
\begin{prop} The linear extension of the map
\begin{eqnarray*}
\phi: (V_{Lie})^{\leq 0}&\to & A(V)_{Lie}\\
 a_{|a|-1} &\mapsto& a+O(V)\quad\quad\; (a_{|a|-1}\in (V_{Lie})^0)\\
 a_n &\mapsto & O(V)\quad\quad\quad\quad (a_n\in (V_{Lie})^-)
 \end{eqnarray*}
is a morphism of Lie algebras.	 \end{prop}

\begin{proof} We have already seen that $\phi$ restricted to $(V_{Lie})^0$ is a morphism of Lie algebras.  Therefore, it suffices to show that the image of the space\[ (V_{Lie})^0\cap [(V_{Lie})^-, (V_{Lie})^-] \] under $\phi$ lies in $O(V)$.
  
Assume that $a_n, b_m\in (V_{Lie})^-$, and $[a_n, b_m]\in (V_{Lie})^0$.  Then it follows that $\wt{a} = \wt{b}=Re(|a|) = Re(|b|)$, and $|a| +|b|\in\bb{Z}$.  We also have $n=m=\wt{a}-1$.  We calculate 
\begin{eqnarray*}
\phi\left([a_{n}, b_m] \right)&=& \phi\left(\left[a_{\wt{a}-1}, b_{\wt{b}-1}\right]\right)\\
&=& \phi\left(\sum_{i\geq 0}{\wt{a}-1\choose i} (a(i)b)_{\wt{a} +\wt{b} -2-i}\right)\\
&=& \phi\left(\sum_{i\geq 0}{\wt{a}-1\choose i} (a(i)b)_{\wt{a(i)b}-1}\right)\\
&=&\sum_{i\geq 0}{\wt{a}-1\choose i} (a(i)b)\\
&=& Res_z\left(Y(a,z)b\;(1+z)^{\wt{a}-1}\right).
\end{eqnarray*}
The lemma then follows from the claim that \begin{equation}
Res_z\left(Y(a,z)b\;(1+z)^{\wt{a}-1}\right)\equiv b\circ a-a\circ b\equiv 0\quad \quad mod\;O(V).\label{LieClaim}\end{equation}
  To see this first congruence, we use the previous lemma and the fact that $|a|+|b|=\wt{a}+\wt{b}$ to calculate $(mod\;O(V))$
\begin{eqnarray*}
a\circ b &=& Res_z\left( Y(a,z)b\;\frac{(1+z)^{\wt{a}-1}}{z}\right) \\
&\equiv & Res_z\left( Y(b, \frac{-z}{1+z})a\;\frac{(1+z)^{\wt{a}-1}}{z}(1+z)^{-\wt{a}-\wt{b}}\right)\\
&=& Res_z\left( Y(b, \frac{-z}{1+z})a\;\frac{(1+z)^{-\wt{b}+1}}{-z}\frac{\partial}{\partial z}\left(\frac{-z}{1+z}\right)\right)\\
&=&Res_w\left(Y(b,w)a\;\frac{(1+w)^{\wt{b}}}{w}\right),
\end{eqnarray*}
where in the last equality we use the formula
\begin{equation*}
Res_w\;g(w) = Res_z\;g(f(z))\frac{\partial}{\partial z}f(z).
\end{equation*}
Then we have $(mod\;O(V))$
\begin{eqnarray*}
a\circ b-b\circ a &\equiv & Res_z\left(Y(a,z)b\;\frac{(1+z)^{\wt{a}-1}}{z}\right) - Res_z\left(Y(a,z)b\;\frac{(1+z)^{\wt{a}}}{z}\right) \\
&=& Res_z\left(Y(a,z)b\;(1+z)^{\wt{a}-1}\right)\left(\frac{1}{z} -\frac{1+z}{z}\right)\\
&=& - Res_z\left(Y(a,z)b\;(1+z)^{\wt{a}-1}\right)
\end{eqnarray*}
which proves \eqref{LieClaim}, and hence the result.\end{proof}

\quad To summarize, we have an action of the Lie subalgebra $(V_{Lie})^{\leq 0} $ on $U$, where $a_{|a|-1}$ acts as $o(a)$, and $ (V_{Lie})^- $ annihilates.  We now consider the induced module 
\[M(U) = Ind_{\mathscr{U}((V_{Lie})^{\leq 0})}^{\mathscr{U}(V_{Lie})}\;U \cong \;\mathscr{U}(V_{Lie})\otimes_{\mathscr{U}((V_{Lie})^{\leq 0})}U\cong S((V_{Lie})^+)\otimes_{\bb{C}}U,\]
where $S(V)$ denotes the symmetric algebra on $V$ and $\mathscr{U}(\cdot)$ denotes the universal enveloping algebra.

\quad Note that $M(U)$ inherits a $\bb{C}$-grading from $S((V_{Lie})^+)$ if we assert that the subspace $U$ has degree 0.  More precisely, we can write 
\[M(U)=M(U)(0)\oplus\bigoplus_{\substack{\mu\in\bb{C}\\Re(\mu)>0}}M(U)(\mu)\]
where $M(U)(0) = U$.

\quad We now define an action of $V$ on $M(U)$ via\begin{equation}\label{action}
Y_{M(U)}(v,z) = \sum_{n\in\bb{Z}}v_n\;z^{-n-1}.\end{equation}
Our goal is to show that a certain quotient of $M(U)$ is an admissible $V$-module.  Note that $(M(U), Y_{M(U)})$ satisfies the requirements to be an admissible $V$-module, except for the Jacobi identity.

\quad It is clear from the commutation relations \eqref{LieCommRlns} that the action \eqref{action} satisfies the weak commutativity relations.  Therefore, to establish the Jacobi identity, it is sufficient to establish the weak associativity.  To do this, we quotient out by the desired relations.

\quad In particular, let $W$ be the subspace of $M(U)$ spanned by the coefficients of $z_0^iz_2^j$ in the expressions
\begin{equation}
(z_0+z_2)^{\wt{a}-1+\delta_{r,0}}Y(a, z_0+z_2)Y(b, z_2)u - (z_2+z_0)^{\wt{a}-1+\delta_{r,0}}Y(Y(a, z_0)b, z_2)u\label{defW}
\end{equation}
 for homogeneous $a\in V^r$, $b\in V$, and $u\in U$.  Now define 
\[\bar{M}(U) = M(U)/\mathscr{U}(V_{Lie})W.\]
Note that $\mathscr{U}(V_{Lie})W$ is a graded $V_{Lie}$-submodule of $M(U)$, and so $\bar{M}(U)$ inherits the $\bb{C}$-grading from $M(U)$.  It is not yet clear that $\bar{M}(U)(0) \neq 0$, but this fact will follow from Theorem \ref{thm1}.
\bigskip
\begin{prop} Let $M$ be a module for $V_{Lie}$ with the property that there is a subspace $U$ of $M$ such that $U$ generates $M$ as a module for $V_{Lie}$ and, for any $a\in V^r$ and any $u\in U$, there is a positive integer $k$ such that 
\begin{equation}
(z_0+z_2)^kY(a, z_0+z_2)Y(b, z_2)u = (z_2+z_0)^kY(a,z_0)Y(b, z_2)u
\end{equation}
for any $b\in V$.  Then $M$ is a weak $V$-module.\end{prop}

\begin{proof} See \cite{DLM2} Proposition 6.1, and take $r=s=0$.\end{proof}

\begin{thm}\label{thm1} The space $\bar{M}(U)$ is an admissible $V$-module.\end{thm}

\begin{proof}  Clearly $\bar{M}(U)$ satisfies the conditions placed on $M$ in the previous proposition, so we find that $\bar{M}(U)$ is a weak $V$-module.  Moreover, by previous remarks, we know that $\bar{M}(U)$ satisfies the grading requirements, except for possibly the requirement that $\bar{M}(U)(0) \neq 0$, but this is demonstrated in the proof of Theorem \ref{thm2}.\end{proof}

\quad We now describe the construction of the $\Lambda$ functor.  Note that $M(U)$ has a maximal graded $V_{Lie}$-submodule $J$ with the property that $J\cap U=0$.  The main result of this section is the following:
\bigskip

\begin{thm} \label{thm2}The space $\Lambda(U)=M(U)/J$ is an admissible $V$-module with the property that the lowest weight space of $\Lambda(U)$ is $U$.\end{thm}

\quad The proof of this theorem relies on several results.  First, it is easy to see that the $V_{Lie}$-submodule $J$ can be described alternatively as
\begin{equation}
J =\{v\in M(U)\;|\;\left< u', xv\right> = 0\;\text{ for all } u'\in U^*,\;x\in\mathscr{U}(V_{Lie})\}\label{charJ}
\end{equation}
where $U^*$ is the set of linear functionals of $U$ which are trivially extended to all of $M(U)$.  Then to prove the theorem, we must show that $\mathscr{U}(V_{Lie})W\subseteq J$.

\quad The next lemma and its proof are slight modifications of Lemma 6.7 of \cite{DLM2}.  We let $u'\in U^*$, and $u\in U$.
\bigskip
\begin{lem} Assume $a\in V^r$ for some nonzero $r$.  Then, for any $i,j\geq 0$, we have
\begin{eqnarray*}
&&Res_{z_0}z_0^{-1+i}(z_0+z_2)^{\wt{a}-1+j}\left< u', Y(a, z_0+z_2)Y(b, z_2)u\right> \\
&&\quad =Res_{z_0}z_0^{-1+i}(z_2+z_0)^{\wt{a}-1+j}\left< u', Y(Y(a, z_0)b, z_2)u\right> 
\end{eqnarray*}
for any $b\in V$.\end{lem}

\begin{proof} Since $j\geq 0$ and $|a|\neq \wt{a}$, we know that $a(\wt{a}-1+j)$ is an element of $(V_{Lie})^-$, and therefore annihilates $u$.  Then, for any $i\geq 0$, we have
\[Res_{z_1}(z_1-z_2)^iz_1^{\wt{a} -1+j}Y(b, z_2)Y(a, z_1)u = 0.\]  Note that we also have the commutator formula
\[ [Y(a, z_1), Y(b, z_2)] = Res_{z_0}z_2^{-1}\delta\left(\frac{z_1-z_0}{z_2}\right)Y(Y(a, z_0)b, z_2)\]
as operators on $M$.  Then
\begin{eqnarray*}
&&Res_{z_0}z_0^i(z_0+z_2)^{\wt{a}-1+j}Y(a, z_0+z_2)Y(b, z_2) u \\
&&\quad =Res_{z_1}(z_1-z_2)^iz_1^{\wt{a}-1+j}Y(a, z_1)Y(b, z_2) u\\
&&\quad =Res_{z_1}(z_1-z_2)^iz_1^{\wt{a}-1+j}[Y(a, z_1),Y(b, z_2) ]u\\
&&\quad =Res_{z_0}Res_{z_1}(z_1-z_2)^iz_1^{\wt{a}-1+j}z_2^{-1}\delta\left(\frac{z_1-z_0}{z_2}\right)Y(Y(a, z_0)b, z_2)u\\
&&\quad =Res_{z_0}Res_{z_1}z_0^iz_1^{\wt{a}-1+j}z_1^{-1}\delta\left(\frac{z_2+z_0}{z_1}\right)Y(Y(a, z_0)b, z_2)u\\
&&\quad =Res_{z_0}z_0^i(z_2+z_0)^{\wt{a}-1+j}Y(Y(a, z_0)b, z_2)u,
\end{eqnarray*}
which proves the lemma for $i\geq 1$.  Now assume that $i=0$.  Then
\begin{eqnarray*}
&&Res_{z_0}z_0^{-1}(z_0+z_2)^{\wt{a}-1+j}\left<u',Y(a, z_0+z_2)Y(b, z_2) u\right> \\
&&\quad =\sum_{k=0}^{\infty}{j\choose k}Res_{z_0}(z_0+z_2)^{\wt{a}-1}z_0^{k-1}z_2^{j-k}\left<u',Y(a, z_0+z_2)Y(b, z_2) u\right> \\
&&\quad =\sum_{k=1}^{\infty}{j\choose k}Res_{z_0}(z_2+z_0)^{\wt{a}-1}z_0^{k-1}z_2^{j-k}\left<u',Y(Y(a, z_0)b, z_2) u\right> \\
&&\quad\quad \quad+Res_{z_0}(z_0+z_2)^{\wt{a}-1}z_0^{-1}z_2^{j}\left<u',Y(a, z_0+z_2)Y(b, z_2) u\right> 
\end{eqnarray*}

The lemma then follows from the claims that
\begin{equation}\label{casei=j=0pt1}
Res_{z_0}z_0^{-1}(z_0+z_2)^{\wt{a}-1}\left<u',Y(a, z_0+z_2)Y(b, z_2) u\right>=0
\end{equation}
and
\begin{equation}\label{casei=j=0pt2}
Res_{z_0}z_0^{-1}(z_2+z_0)^{\wt{a}-1}\left<u', Y(Y(a, z_0)b, z_2)u\right>=0.
\end{equation}

\quad To see \eqref{casei=j=0pt1}, recall that $b(\wt{b}-1-n)u = 0$ since $b_{\wt{b}-1+n}\in(V_{Lie})^-$ for $n\leq 1$.  Then we have 
\[\left< u', Y(a, z_0+z_2)Y(b, z_2)u\right> = \sum_{n\geq 0}\left< u', a(\wt{a}-1+n)b(\wt{b}-1-n)u\right>(z_0+z_2)^{-\wt{a}-n}z_2^{-\wt{b}+n}.\]
This gives
\begin{eqnarray*}
&&Res_{z_0}z_0^{-1}(z_0+z_2)^{\wt{a}-1}\left<u',Y(a, z_0+z_2)Y(b, z_2) u\right> \\
&&\quad =Res_{z_0}\sum_{n\geq 0}\left< u', a(\wt{a}-1+n)b(\wt{b}-1-n)u\right>(z_0+z_2)^{-n-1}z_2^{-\wt{b}+n-1}\\
&&\quad =Res_{z_0}\sum_{n\geq 0}\sum_{k\geq 0}{-n-1\choose k}\left< u', a(\wt{a}-1+n)b(\wt{b}-1-n)u\right>z_2^{-\wt{b}+n-1+k}z_0^{-n-1-k}\\
&&\quad =\left< u', a(\wt{a}-1)b(\wt{b}-1)u\right>z_2^{-\wt{b}-1}\\
&&\quad = 0,
\end{eqnarray*}
where the last equality holds because $a\in V^r$, so $a(\wt{a}-1)$ annihilates $b(\wt{b}-1)u$.  This proves \eqref{casei=j=0pt1}.

\quad An analogous calculation shows that 
\begin{eqnarray*}
&&Res_{z_0}z_0^{-1}(z_2+z_0)^{\wt{a}-1}\left<u', Y(Y(a, z_0)b, z_2)u\right>\\
&&\quad = \sum_{k\geq 0} {\wt{a}-1\choose k}\left<u', (a(k-1)b)(|a(k-1)b|-1)u\right>z_2^{-|a|-|b|+\wt{a}-1},\\
\end{eqnarray*}
where we have made the assumption that $\wt{a(j)b} = |a(j)b|$ for $j\in\bb{Z}$.  This assumption is valid because otherwise the element $a(j)b$ would have no zero mode, which would imply that 
\[\left< u', (a(j)b)(k)u\right> = 0\]
for any $k\in\bb{Z}$.
Then \eqref{casei=j=0pt2} follows from the fact that
\[\sum_{k\geq 0} {\wt{a}-1\choose k} (a(k-1)b)(|a(k-1)b|-1)u = o(a\circ b)u =0\]
since $O(V)$ annihilates $U$.  \end{proof}
\bigskip

\begin{lem}  Assume $a\in V^r$ for some nonzero $r$ and $j\geq 0$, but $i\in \bb{Z}$.  Then we have
\begin{eqnarray*}
&&Res_{z_0}z_0^{-1+i}(z_0+z_2)^{\wt{a}-1+j}\left< u', Y(a, z_0+z_2)Y(b, z_2)u\right> \\
&&\quad =Res_{z_0}z_0^{-1+i}(z_2+z_0)^{\wt{a}-1+j}\left< u', Y(Y(a, z_0)b, z_2)u\right> 
\end{eqnarray*}
for any $b\in V$.\end{lem}

\begin{proof} The result holds for $i\geq -1$ by the previous lemma.  The proof then follows from the $T$ derivative property together with induction on $i$.  A complete proof is given in \cite{DLM2} Lemma 6.8.\end{proof}

\bigskip
\begin{prop} The following holds for any homogeneous $a\in V^r$, $b\in V$, $u'\in U^*$, $u\in U$ and $j\geq 0$:
\begin{eqnarray}
&&\left< u', (z_0+z_2)^{\wt{a}-1+\delta_{r,0}+j}Y_{M(U)}(a, z_0+z_2)Y(b, z_2)u\right>\notag\\
&&\quad\quad=\left< u', (z_2+z_0)^{\wt{a}-1+\delta_{r,0}+j}Y_{M(U)}(Y(a, z_0)b, z_2)u\right>.\label{prop6.5}
\end{eqnarray}\end{prop}

\begin{proof} Assume that $b\in V^s$ for some $s$.  If $r=0$ and $s\neq 0$, then both sides of \eqref{prop6.5} are equal to 0.  This follows from the fact that the modes of $a$ act on $V^0$, but the modes of $b$ do not, i.e., for all $i\in\bb{Z}$,
\[b(i)V^0\cap V^0 = 0.\]
\quad Now assume that $r = s = 0$.  Then this is essentially the case where $V=V^0$, so we appeal to \cite{DLM2} Proposition 6.5.  The only remaining case is where $r\neq 0$, and this case follows from the previous two lemmas.\end{proof}
\bigskip

\begin{prop} Assume that $M$ is a $V_{Lie}$-module and $W$ is a subspace of $M$ such that $W$ generates $M$ as a $V_{Lie}$-module.  Assume further that for any $a\in V^r$, and any $w\in W$, that these is some $k\in\bb{Z}$ such that
\begin{eqnarray}
&&\left< w', (z_0+z_2)^kY_{M(U)}(a, z_0+z_2)Y(b, z_2)w\right>\notag\\
&&\quad\quad=\left< w', (z_2+z_0)^kY_{M(U)}(Y(a, z_0)b, z_2)w\right>.\label{prop6.9}
\end{eqnarray}
for any $b\in V$ and any $w'\in W^*$.  Then \eqref{prop6.9} holds for any $u\in M$.\end{prop}

\begin{proof} See \cite{DLM2} Proposition 6.9.\end{proof}
\bigskip
\begin{prop} Let $M$ be as in the previous proposition.  Then for any $x\in\mathscr{U}(V_{Lie})$, $a\in V^r$, $u\in M$, there is an integer $k$ such that 
\begin{eqnarray}
&&\left< w', (z_0+z_2)^kx\cdot Y_{M(U)}(a, z_0+z_2)Y(b, z_2)w\right>\notag\\
&&\quad\quad=\left< w', (z_2+z_0)^kx\cdot Y_{M(U)}(Y(a, z_0)b, z_2)w\right>.\label{prop6.10}
\end{eqnarray}
for any $b\in V$ and any $w'\in W^*$.\end{prop}

\begin{proof} See \cite{DLM2} Proposition 6.10.\end{proof}

\quad We now return to the proof of Theorem \ref{thm2}.  We let $M=M(U)$ and $W=U$ in the previous proposition.  In particular, from \eqref{charJ} together with the previous three propositions, and from the definition \eqref{defW} of $W$, we see that $\mathscr{U}(V_{Lie})(W)\subseteq J$.  Thus, $\Lambda(U)$ is an admissible module for $V$, and it satisfies the property that its vacuum space is equal to $U$, i.e., $\Omega(\Lambda(U)) = U$. (Indeed, if $v$ were a lowest weight vector in $M(U)$ but $v$ was not in $U$, then $v$ would generate a submodule $N$ which would satisfy $N\cap U =0$, hence $N\subseteq J$.)  This completes the proof of Theorem \ref{thm2}.

\quad Note that we may now complete the proof of Theorem \ref{thm1} by declaring that the lowest weight space of $\bar{M}(U)$ is equal to $U$.  This follows from the fact that $\mathscr{U}(V_{Lie})(W)\subseteq J$.
\bigskip
\begin{thm} The functors $\Lambda$ and $\Omega$ induce mutually inverse bijections on the isomorphism classes of the categories of simple $A(V)$-modules and simple admissible $V$-modules.\end{thm}

\begin{proof} First, Theorem \ref{thm2} implies that $\Omega(\Lambda(U)) = U$ for any $A(V)$-module $U$.  Now assume that $M$ is a simple $V$-module.  Then we know that $\Omega(M)$ is a simple $A(V)$-module (Proposition \ref{prop3.12}).  Moreover, we know that $\Lambda(\Omega(M))$ is a simple $V$-module (See \cite{DLM2} Lemma 7.1).  Then $M$ and $\Lambda(\Omega(M))$ are both simple quotients of the universal object $\bar{M}(\Omega(M))$, and are therefore isomorphic, since $\bar{M}(\Omega(M))$ has a unique maximal ideal $J$ subject to $J\cap\Omega(M) = 0$.\end{proof}

\section{Pseudo VOAs}

The goal of this section is to construct a family of $\bb{C}$-graded vertex algebras from a given VOA.  Since this class of examples come from a VOA, there is some additional conformal structure which is not present in an arbitrary $\bb{C}$-graded vertex algebra.  Thus we have the following definition: 
\bigskip
\begin{definition*} A \emph{pseudo vertex operator algebra}, or PVOA, is a $\bb{C}$-graded vertex algebra $V= \bigoplus_{\mu\in\bb{C}}V_{\mu}$ with the following additional properties:

\quad (i) There is a vector $\omega\in V_2$ called the \emph{conformal vector} such that the operators $\{L(n)\}$ defined by $Y(\omega ,z) = \sum_{n\in\bb{Z}}L(n)\;z^{-n-2}$ generate the Virasoro algebra,

\quad (ii) $L(-1) = T$,

\quad (iii) For any $v\in V_{\mu}$, there is some $n\in \bb{N}$ such that $(L(0)-\mu)^nv=0$, and $dim(V_{\mu})<\infty$ for all $\mu\in\bb{C}$,

\quad (iv) $Re(\mu) \geq |Im(\mu)|$ for all but finitely many $\mu\in Spec_V(L(0))$.

This completes the definition.
\end{definition*}

\quad One notes that any VOA is a PVOA, and any PVOA contains the VOA generated by $\omega$.  In fact, a PVOA satisfies the axioms of a VOA with the exception of the semisimplicity of $L(0)$ and the integrality of the spectrum of $L(0)$.  The conformal structure of a PVOA allows us to define a new class of $V$-module.
\bigskip
\begin{definition*} A weak $V$-module $M$ is called \emph{ordinary} if $M$ is has a $\bb{C}$-grading induced by $L(0)$-eigenvalues\[M=\bigoplus_{\mu\in Spec_M(L(0))}M(\mu),\]such that each graded subspace $M(\mu)$ is finite dimensional.  We also require that $Re(\mu) >0$ for all but finitely many $\mu\in Spec_M(L(0))$. \end{definition*}

\quad One notes that an ordinary $V$-module is an admissible module for $V$ as a $\bb{C}$-graded vertex algebra.

\quad Throughout this section, we let \[V = \bb{C1}\oplus\bigoplus_{n\in\bb{N}}V_n\] be a CFT-type VOA of central charge $c$.  The next proposition shows how one may ``deform" the conformal structure on a VOA.

\bigskip
\begin{prop} Let $h\in V_1$ with $h(1)h  = \alpha\bb{1}$ and $L(1)h = \beta\bb{1}$.  Then the modes of \[\omega_h = \omega +L(-1)h\] satisfy the Virasoro relations with central charge $c_h =c+12(\beta-\alpha)$.
\end{prop}

\begin{proof} We let $Y(\omega_h, z) = \sum_{n\in\bb{Z}} L_h(n)z^{-n-2}$, so that 
\begin{equation}
L_h(n) = L(n) - (n+1)h(n).\label{shiftedmodes}
\end{equation}
A standard calculation shows that 
\begin{equation*}
[L_h(m), L_h(n)] = (m-n)L_h(m+n) +\frac{m^3 - m}{12}\delta_{m+n,0}(c+ 12(\beta-\alpha))Id. 
\end{equation*}
\end{proof}

\quad In what follows, we will concern ourselves primarily with the case where $h$ is a \emph{primary} vector, i.e., $L(1) h =0$, so that $\beta = 0$ in the previous proposition.  Note that by \eqref{shiftedmodes} we have $L_h(-1) = L(-1)$ and $L_h(0) = L(0) - h(0)$.  We denote by $V^h$ the quadruple $(V, Y, \bb{1}, \omega_h)$.  In general, it is not clear whether $V^h$ is a PVOA, since the spectrum of $L_h(0)$ is not well understood.

\subsection{Strongly Regular VOAs and Jacobi Forms}

\begin{definition*} A VOA is called \emph{strongly regular} if $V$ is rational, $C_2$-cofinite, CFT-type, and $L(1)V_1=0$.
\end{definition*}

The following two theorems are proved in \cite{M}:
\bigskip

\begin{thm} Let $V$ be a strongly regular VOA.  Then the weight one subspace $V_1$ is a reductive Lie algebra.
\end{thm}
\quad This theorem allows us to work with a Cartan subalgebra of $V_1$.  Indeed, a reductive Lie algebra $\mf{g}$ can be written as
\begin{equation*}
\mf{g} = \mf{g}_{ss} \oplus \mf{a}
\end{equation*}
where $\mf{g}_{ss}$ is a semisimple Lie algebra and $\mf{a}$ is an abelian ideal of $\mf{g}$.  In this case, if $\mf{h}$ is a Cartan subalgebra of $\mf{g}_{ss}$, then $\mf{h}\oplus\mf{a}$ is a Cartan subalgebra of $\mf{g}$.

\bigskip

\begin{thm}\label{proph} Let $V$ be a strongly regular VOA.  Then any Cartan subalgebra $H\subseteq V_1$ has a basis $\{h_1, ...\; ,h_r\}$ such that each $h_i$ satisfies the following two properties:

\quad\quad\quad\quad (i) $h_i(0)$ is a semisimple operator with integral eigenvalues

\quad\quad\quad\quad (ii) $\left<h_i, h_i\right> \in2\bb{Z}$

\end{thm}

\quad Now we consider a strongly regular simple VOA $V$.  In this case, we denote by $\{M_1=V,.. .,M_k\}$ the set of inequivalent simple admissible $V$-modules.  For each $i\in\{1,.. .k\}$, we define
\[J^i(\tau,z) = Tr_{M_i}q^{L_i(0)-c/24}\zeta^{h(0)},\]
where $h$ is assumed to be an element in $V_1$ with the property that $h(0)$ acts semisimply on every $M_i$ with integral spectrum.  We note that $J^i(\tau, z)$ is, up to an overall shift, a power series in $q$:
\begin{equation}\label{jacpows}
J^i(\tau, z)= q^{s_i}\sum_{n=0}^{\infty}\sum_{r\in\bb{Z}}c^i(n,r)q^n\zeta^r.\end{equation}
where $s_i = -c/24 +\lambda_i$ and $\lambda_i$ is the conformal weight of $M_i$.  We have the following 
result (\cite{KM}):
\bigskip

\begin{thm} Let $V$ be a strongly regular simple VOA.  Let $h\in V_1$ be an element such that $h(0)$ is semisimple on every $M_i$ with integral spectrum.  Then the functions $J^i(\tau,z)$ are holomorphic in $\bb{H}\times \bb{C}$, and the following functional equations hold for all $\gamma = \begin{pmatrix}a&b\\c&d\end{pmatrix}\in SL(2,\bb{Z})$ and $(u,v)\in\bb{Z}^2$:

\quad (i) There are scalars $a_{ij}(\gamma)$ depending only on $\gamma$ such that 
\[J^i\left(\gamma\tau, \frac{z}{c\tau+d}\right) = e^{\pi i c z^2\left<h,h\right> /(c\tau +d)}\sum_{j=1}^ra_{ij}(\gamma)J^j(\tau,z),\]

\quad (ii) There is a permutation $i\mapsto i'$ of $\{1,.. .,k\}$ such that 
\begin{equation}\label{vvjac}J^i(\tau, z+u\tau+v) = e^{-\pi i\left< h,h\right>(u^2\tau+2uz)}J^{i'}(\tau,z).\end{equation}\end{thm}

\quad We can use the tranformation property \eqref{vvjac} to deduce information about the coefficients $c^i(n,r)$.  For each $i\in\{1,.. ., k\}$,  set $d_i = \max_{j\in\{1,.. .,k\}}|s_i-s_j|$, and set $m = \left<h,h\right>/2$.  Then we have 
\bigskip
\begin{prop}\label{quadbnd} With the previous notation, one has $c^i(n,r) = 0$ if
\[r^2 > m^2 + 4m(n+d_i).\]\end{prop}

\begin{proof} The transformation property \eqref{vvjac} implies that \[c^{i'}(n +ru +mu^2 +s_i-s_{i'},r+2um) = c^i(n,r).\]  Then using \eqref{jacpows}, we see that $c^i(n,r) = 0$ if $n+ru+mu^2 +s_i-s_{i'}<0$.  From this we can see that $c^i(n,r) = 0$ whenever there is an integer $u$ such that $(n+d_i)+ru+mu^2 <0$.

\quad The condition $(n+d_i)+ru +mu^2 <0$ is equivalent to the condition that the quadratic polynomial $f(x) = (n+d_i)+rx+mx^2$ has a negative value when evaluated at some integer $u$.  Since $m$ is a positive number, elementary algebra tells us that this condition is satisfied if the roots of $f$ are more than 1 unit apart, i.e., if \[\dfrac{\sqrt{r^2 -4m(n+d_i)}}{2m} \;>\;\dfrac{1}{2},\]
which is equivalent to $r^2 > m^2 +4m(n+d_i)$.\end{proof}

\quad Note that the coefficient $c^i(n,r)$ is the dimension of the space
\[M_i(n,r) = \{\;u\in M_i\;|\; L(0) u =nu\quad\text{and}\quad h(0)u = ru\;\}.\]
In particular, one has \[c^1(n,r) = Dim\{ \;v\in V_n \;|\;h(0)v = rv\;\}.\]

Now define $h^n\in\bb {Z}_{\geq 0}$ to be the largest element of the set
\begin{equation*}
\{ \quad |m| \quad \vert \quad m \text{ is an eigenvalue of }h(0)\text{ on }V_n\;\},
\end{equation*}
where $|m |$ denotes the absolute value of $m$.  In other words, for each fixed $n$, $h^n$ denotes the absolute value of the $r$ of largest absolute value for which $c^1(n,r)\neq 0$.

\bigskip

\begin{lem}\label{ppp} Let $V$ and $h$ be as above.  Then the following inequality holds for every $n$:
\[(h^n)^2\leq n(4m) + m^2+4md_1.\]
In other words, $h^n\sim O(\sqrt{n})$ as $n\to \infty$.
\end{lem}

\begin{proof}By Proposition \ref{quadbnd}, it follows that $(h^n)^2\leq m^2+4m(n+d_1)$ for every $n$.\end{proof}

\quad Of course, one obtains a similar statement if we replace $h$ by any complex scalar multiple $\lambda h$ of $h$.  In this case, the previous proposition says that $|\lambda h^n| \sim O(\sqrt{n})$.  This leads us to the following extension of Proposition \ref{ppp}:

\bigskip

\begin{lem}\label{bigo} Let $V$ be a simple strongly regular VOA and let $h\in V_1$ be such that $h(0)$ is semisimple.  Then
$|h^n|\sim O(\sqrt{n})$ as $n\to \infty$.\end{lem}

\begin{proof} Since $h$ is a semisimple element of $V_1$, we know that $h$ is an element of some Cartan subalgebra $H$ of $V_1$.  Now using a basis $\{ h_1,.. .,h_k\}$ of $H$ as in Theorem \ref{proph}, we write
\begin{equation*} 
h=\sum_{i=1}^{k}\lambda_ih_i, 
\end{equation*} 
for some $\lambda_i\in\bb{C}$, so that 
\begin{equation} \label{h0decomp}
h(0)=\sum_{i=1}^{k}\lambda_ih_i(0).
\end{equation} 
Since $\{h_i(0)\}$ is a set of commuting, semisimple operators, it follows that 
\begin{equation*}
0\leq |h^n|\leq \sum_{i=1}^k |\lambda_ih_i^n|.
\end{equation*}
Our prior remarks show that each of the above summands satisfies $|\lambda_ih_i^n| \sim O(\sqrt{n})$, and so it follows that $|h^n|\sim O(\sqrt{n})$.\end{proof}

\bigskip

\begin{lem}\label{lhspec} Let $V$ and $h$ be as in Proposition \ref{bigo}, and consider $L_h(0) = L(0) - h(0)$.  Then $Re(\mu) \geq |Im(\mu)|$ for all but finitely many $\mu \in Spec_V(L_h(0))$.\end{lem}

\begin{proof} Making use of the decomposition \eqref{h0decomp}, one sees that 
\begin{equation*}
Spec_{V}(L_h(0))\;\subseteq\; \bb{Z} -\sum_{i=1}^k\lambda_i\bb{Z}\;\subset\;\bb{C}.
\end{equation*}
Any $\mu\in Spec_{V}(L_h(0))$ is of the form 
\[\mu = n-\sum_{i=1}^k\lambda_ia_i,\]
where $a_i\in Spec_{V_n}(h_i(0))\subset\bb{Z}$.

\quad If $h(0)$ does not act as $0$ on $V_n$, then $h_i^n >0$ for some $i$.  Without loss of generality we may assume $h_1^n>0$.  Now choose $m\in\bb{R}$ such that 
\begin{equation*}
mh_1^n \;>\; k\cdot \max_{i}|Re(\lambda_ih_i^n)| + k\cdot \max_i |Im(\lambda_ih_i^n)|.
\end{equation*}

Using Lemma \ref{bigo}, we see that that $n-mh^n_1 >0$ for all but finitely many $n$.  Therefore, since $h_1^n >0$, we have
\begin{equation*}
n-k\cdot \max_{i}|Re(\lambda_ih_i^n)| - k\cdot \max_i |Im(\lambda_ih_i^n)| \;>\; n-mh_1^n\; >\;0
\end{equation*}
for all but finitely many $n$.  In particular,
\begin{equation*}
n-k\cdot \max_{i}|Re(\lambda_ih_i^n)| \;>\; k\cdot \max_i |Im(\lambda_ih_i^n)|.
\end{equation*}
for all but finitely many $n$.

Then we have
\begin{eqnarray*}
Re\left(n-\sum_{i=1}^k \lambda_ia_i \right) &=&n-Re\left(\sum_{i=1}^k \lambda_ia_i \right) \\
&\geq& n- \left|Re\left(\sum_{i=1}^k \lambda_ia_i \right)\right| \\
&\geq & n-k\cdot \max_i|Re(\lambda_i h_i^n)| \\
&>& k\cdot \max_i|Im(\lambda_ih_i^n)| \\ 
&\geq &\sum_{i=1}^k | Im(\lambda_ih_i^n)| \\
&\geq &\sum_{i=1}^k | Im(\lambda_ia_i)| \\
& \geq & \left|Im\left(\sum_{i=1}^r\lambda_ia_i\right)\right|
\end{eqnarray*}
for all but finitely many $n$.  Since $h(0)$ has only finitely many eigenvalues on each $V_n$, it follows that the above inequality holds for all but finitely many $L_h(0)$-eigenvalues $\mu$.  \end{proof}

\bigskip
\begin{lem}\label{findim} Let $V$ and $h$ be as in Proposition \ref{bigo}.  Then the $L_h(0)$ eigenspaces are all finite dimensional.\end{lem}

\begin{proof} Assume to the contrary that $\mu$ is the eigenvalue corresponding to an infinite dimensional $L_h(0)$-eigenspace.  Since each $V_n$ is of finite dimension, it follows that there must be an infinite number of $k\in\bb{N}$ for which \[k-\mu \in Spec_{V_k}(h(0)),\]
which is impossible due to Lemma \ref{bigo}.\end{proof}

\bigskip
\begin{thm}\label{pvoa} Let $V$ be a simple strongly regular VOA.  Then $V^h$ is a simple PVOA for any $h\in V_1$. \end{thm}

\begin{proof}  We first remark that simplicity of $V^h$ is equivalent to simplicity of $V$ since they are both the same underlying vertex algebra.  The $C_2$-cofiniteness of $V$ implies that $V$ is finitely generated, that is, 
\begin{equation*}
V = \left<\bigoplus_{n=0}^{N}V_n\right>.
\end{equation*}  Then the action of $h(0)$ on $V$ is completely determined by the action of $h(0)$ on $\bigoplus_{n=0}^{N}V_n$, which is a finite dimensional $V_1$-module.  Therefore, we consider the abstract Jordan decomposition $h(0)=h^{ss}(0) + h^n(0)$, with $h^{ss}(0)$ a semisimple operator and $h^n(0)$ a nilpotent operator such that $h^{ss}(0)$ and $h^n(0)$ commute.  The theory of Lie algebras then ensures that the operators $h^{ss}(0)$ and $h^n(0)$ are modes of elements $h^{ss}$, $h^n\in V_1$.  Thus any $h\in V_1$ decomposes as $h=h^{ss} + h^n$, where $h^{ss}(0)$ is a semisimple operator on $V$ and $h^n(0)$ is a nilpotent operator on $V$ that commutes with $h^{ss}(0)$.  

\quad First assume that $h^n = 0$.  Then $L_h(0) = L(0) - h^{ss}(0)$ is a semisimple operator.  In this case $V^h$ has a grading 
\[V^h = \bigoplus_{\mu\in Spec_{V^h}(L(0))}V_{\mu}\] which satisfies \eqref{cgraded}.  Moreover, we know that $\bb{1}\in V_0$, so $V_0\neq 0$.  This fact, together with Lemma \ref{lhspec}, implies that $V^h$ has a lowest weight space $V^h_{\lambda}$.  Since $V^h$ is simple as a vertex algebra, it is generated by $V^h_{\lambda}$.  This shows that $V^h$ is a $\bb{C}$-graded vertex algebra.  It is clear from \eqref{shiftedmodes} that $L(-1)=L_h(-1)$.  Therefore, recalling Lemmas \ref{findim} and \ref{lhspec}, we see that $V^h$ is a PVOA.

\quad If $h^n$ is not zero, then $h^n(0)$ simply acts on each $L_h^{ss}(0)$-eigenspace since $[h^n(0), L_h^{ss}(0)]=0$, and so we see that the generalized $L_h(0)$-eigenspaces are precisely the $L_h^{ss}(0)$-eigenspaces.  The finite dimensionality of each $L_h^{ss}(0)$-eigenspace $V^h_{\lambda}$ implies that $(L_h(0)-\lambda)$ is a nilpotent operator on $V^h_{\lambda}$.  The result follows.\end{proof}

\subsection{Regularity of Lattice PVOAs}

Let $V_L$ be a the VOA associated to a rank $k$ positive definite even lattice $L$ (see \cite{FLM}).  Dong \cite{D} proved that $V_L$ is a rational VOA, and that the irreducible ordinary modules for $V_L$ correspond to cosets of $L$ in its dual lattice $L^{\circ}$.  In particular, the irreducible ordinary $V_L$-modules are of the form
\begin{equation*}
V_{L-\lambda} = M(1)\otimes \bb{C}[L-\lambda],
\end{equation*}
where $\lambda\in L^{\circ}$ and $\bb{C}[L-\lambda]$ is the corresponding module for the twisted group algebra $\bb{C}\{ L\}$.

We recall Theorem 3.16 of \cite{DLM1}:
\bigskip
\begin{thm} Let $L$ be any positive definite even lattice.  Then any weak $V_L$-module is completely reducible, and any simple weak $V_L$-module is isomorphic to $V_{L-\lambda}$ for some $\lambda$ in $L^{\circ}$.  In other words, $V_L$ is regular.\end{thm}

\quad Since $V_L$ is a simple, strongly regular VOA, we know that for any $h\in V_1$, one obtains a PVOA by shifting the conformal structure on $V_L$ by the element $h$.  Specifically, we have the PVOA $V_L^h=(V_L,Y,\bb{1},\omega_h)$, where $\omega_h=\omega-L(-1)h$.  

\bigskip

\begin{thm}\label{thm4} The PVOA $V_L^h$ is regular.  In other words, any weak $V_L^h$-module is a direct sum of simple ordinary $V_L^h$-modules.  \end{thm}

\quad Before beginning the proof of Theorem \ref{thm4}, we recall some details about partition functions.  For any coset $L-h$ of $L$ in $H$, we define the formal sum
\[\theta_{L-h}(q) = \sum_{\alpha \in L-h}q^{\left<\alpha,\alpha\right>/2} = q^{\left<h,h\right>/2}\sum_{\alpha\in L}q^{\left<\alpha,\alpha\right>/2 - \left<h,\alpha\right>}\]

We define the partition function of a $V_L^h$-module $M$ as
\[Z_{M,V_L^h}(q) = Tr_{M}q^{L_h(0) - c_h/24}.\]
This expression is to be treated only as a formal sum, since it may contain complex powers of $q$.  Since we are primarily interested in the spectrum of $L_h(0)$  we only need to consider the operator $L^{ss}_h(0) = L(0) - h^{ss}(0)$.  Therefore, we assume that $h^n=0$ (see the proof of Theorem \ref{pvoa}).  For $\lambda\in L^{\circ}$, we consider the partition function of the $V_L^h$-module $V[L-\lambda]$.  If $u\in M(1)(n)$ and $\alpha \in L-\lambda$, we have
\begin{equation}
(L(0)-h(0))(u\otimes e^{\alpha}) = \left(n +\frac{1}{2}\left< \alpha,\alpha\right>-\left<h,\alpha\right>\right)(u\otimes e^{\alpha}).
\end{equation}
As a formal sum, we have that \[Tr_{V[L-\lambda]} q^{L_h(0)} = \sum_{\mu\in \bb{C}}dim(V[L-\lambda](\mu))\;q^{\mu},\] where $V[L-\lambda](\mu)$ is the $L_h(0)$-eigenspace with eigenvalue $\mu$.  Then we calculate:
\begin{eqnarray*}
Z_{V[L-\lambda],V_L^h} &=& Tr_{V[L-\lambda]}q^{L_h(0) - c_h/24}\\
&=&Tr_{V[L-\lambda]}q^{L(0)-h(0) - k/24 + \left<h,h\right>/2}\\
&=&Z_{M(1),V_L}(q)\cdot Tr_{V[L-\lambda]}q^{L(0)-h(0) + \left<h,h\right>/2}\\
&=&Z_{M(1),V_L}(q)\cdot q^{\left<h,h\right>/2}\sum_{\alpha\in\bb{C}[L-\lambda]}q^{\left<\alpha,\alpha\right>/2 - \left<h,\alpha\right>}\\
&=&Z_{M(1),V_L}(q)\cdot\theta_{L-h-\lambda}(q) 
\end{eqnarray*}

\emph{Proof of Theorem \ref{thm4}.} Let $M$ be any weak $V_L^h$-module.  Since $V_L^h$ has the same modes as $V_L$, it follows that $M$ is a weak module for $V_L$.  Due to the regularity of $V_L$, $M$ must decompose as a direct sum of simple ordinary $V_L$-modules.  By \cite{D} we know that each simple ordinary $V_L$-module must be of the form $V_{L-\lambda}$ for some $\lambda\in L^{\circ}$.  Thus, $M$ decomposes as a sum of simple weak $V_L^h$-modules, each of which is of the form $V_{L-\lambda}$.  Moreover, we calculated above that the partition function of $V_{L-\lambda}$ as a $V_L^h$-module satisfies
\[Z_{V[L-\lambda],V_L^h} =Z_{M(1),V_L}(q)\cdot\theta_{L-h-\lambda}(q) ,\]
and this is sufficient to show that $V_{L-\lambda}$ satisfies the grading requirements for ordinary $V_L^h$-modules.  Thus, $M$ is a sum of simple ordinary $V_L^h$-modules.\hfill$\square$

\bigskip
\begin{cor}  $A(V_L^h)$ is a finite dimensional semisimple algebra.\end{cor}

\
\end{document}